\documentclass[preprint,12pt]{article}

\usepackage{amsthm,amsmath,amssymb}
\usepackage{natbib}
\pdfoutput=1
\usepackage[toc,page]{appendix}
\usepackage{xcolor}
\usepackage{tikz}
\usepackage{comment}
\usepackage{float}
\usepackage{bm}
\usepackage{graphicx}
\usepackage{dsfont}
\usepackage{enumerate}
\usepackage{pst-all}
\usepackage[margin=0.75in]{geometry}
\usepackage{pgfplots}

\pgfplotsset{width=10cm,compat=1.9}
\usetikzlibrary{patterns}
\usepgfplotslibrary{external}

\newtheorem{thm}{Theorem}
\newtheorem{prop}{Proposition}
\newtheorem{lem}{Lemma}

\newcommand{\ind}[1]{\,\mathds{1}_{\{#1\}}}

\newcommand{\Var}{{\rm Var}}

\title{Irregularly observed long-memory L\'evy-driven moving average processes}

\author{
Mohamedou Ould Haye$^1$\footnote{Corresponding author: mohamedouhaye@cunet.carleton.ca}
\,and Anne Philippe$^2$
}

\date{
$^1$\small School of Mathematics and Statistics.\\
Carleton University, 1125 Colonel By Dr. Ottawa, ON, Canada, K1S 5B6\\
$^2$ Nantes Universit\'e, CNRS, Laboratoire de Math\'ematiques Jean Leray, LMJL,
UMR 6629, F-44000 Nantes, France.
}

\begin{document}

\maketitle

\begin{abstract}
We study long-memory continuous-time moving-average processes driven by a L\'evy process and observed at random renewal times. The sampling scheme introduces an additional source of randomness through irregular observation times. We establish the asymptotic behaviour of normalized partial sums under both finite- and infinite-mean renewal sampling. 
\end{abstract}

\section{Introduction}\label{Introduction}

Irregularly observed time series arise naturally in many fields, including astronomy \cite{10.1093rastirzac011}, signal processing \cite{SUN2024110075}, environmental sciences \cite{Beelaerts2010TimeSeriesRF}, and biomedical sciences \cite{Shukla2018ModelingIS}. More recently, irregular observation schemes have also been used in differential privacy procedures to enhance data confidentiality; see \cite{pmlr-v151-koga22a}. A common approach to modeling irregularly observed time series is to embed the observations in a continuous-time process sampled at deterministic or random time points.

Random observation schemes can substantially alter some of the probabilistic properties of the underlying continuous-time process. In particular, the dependence structure, joint distributions, and asymptotic behaviour may all be affected by the randomness of the observation mechanism.

Discrete-time linear processes constitute a large body of the time series literature.
Their continuous-time counterparts are L\'evy-driven continuous-time moving average processes. See for instance \cite{Brockwell2001} for continuous-time ARMA processes and \cite{BrockwellFerrazzanoKlueppelberg2013,CohenLindner2013,BAI20161036} for a more general context. 
These processes are defined as stochastic integrals of a deterministic function, called the kernel, with respect to a L\'evy process. More precisely, we consider processes of the form
\begin{equation}\label{CMA}
X_t=\int_{-\infty}^{t} a(t-s)\,dL_s, \qquad t\in {\bf R}^{+},
\end{equation}
where $L=(L_t)_{t\in {\bf R}}$ is a centered L\'evy process with finite second moment satisfying
\begin{equation*}
{\rm Var}(L_1)=1,
\end{equation*}
and where the kernel $a$ has support on ${\bf R}^{+}$ and satisfies
\begin{equation}\label{finite 2}
\int_0^\infty a^2(t)\,dt<\infty.
\end{equation}
In this paper, we focus on the class of long-memory kernels given by
\begin{equation}\label{a}
a(t)\sim c(d)(1+t)^{d-1}, \qquad t\to\infty, \qquad c(d)\neq0,
\end{equation}
where $0<d<1/2$. Under this condition, $(X_t)$ exhibits long-range dependence, with memory parameter $d$.\\
We study the sampled process $X_t$ at random time points defined by
\begin{equation}\label{sampled1}
Y_k=X_{T_k}, \qquad k=1,2,\ldots.
\end{equation}
where  
 $(T_k)$ denotes a renewal process
 \begin{equation*}
T_0=0, \qquad T_k=\sum_{i=1}^{k}\Delta_i,
\end{equation*}
and $(\Delta_i)$ are iid real-valued positive random variables. 
In the financial literature, processes obtained through random time changes \eqref{sampled1} are commonly referred to as time-changed processes; see, for example, \cite{GemanAne2000,GemanMadanYor2001}. The process $(Y_k)$ may be viewed as a discrete-observation analogue of such models, where the underlying process is observed only at the random epochs of a renewal process.

Several aspects of renewal sampling for long-memory processes have recently been investigated. For example, \cite{PhilippeRobetViano} showed that renewal sampling generally destroys Gaussianity. Moreover, if the renewal process has finite mean, suitably normalized partial sums of the sampled process still converge to fractional Brownian motion. Heavy-tailed renewal schemes were further investigated in \cite{philippe-viano2008}. In particular, they show that the memory parameter can be reduced if the renewal process has infinite moment. Statistical inference and covariance asymptotics for sampled long-memory processes were studied in \cite{OuldHaye2023}.

Throughout the paper we assume that the renewal increments satisfy one of the following conditions.

\begin{enumerate}
\item[(A1)] Finite-mean case:
$$
{\bf E}(\Delta_1)<\infty.
$$

\item[(A2)] Heavy-tailed case:
$$
P(\Delta_1>x)=x^{-\alpha}\ell(x), \qquad x\ge1,
$$
where $0<\alpha\le1$ and $\ell$ is slowly varying at infinity, and that ${\bf E}(\Delta_1)=\infty$.
\end{enumerate}
In this paper we study the asymptotic behaviour of the partial sums
\begin{equation}\label{Sndef}
S_n=
\sum_{i=1}^{n}Y_i =
\sum_{i=1}^{n}X_{T_i}
=\int_{-\infty}^{\infty}
\sum_{i=1}^{n}a(T_i-s)\,dL_s,
\end{equation}
where the kernel $a$ is given by \eqref{a}. \cite{CohenLindner2013} studied \eqref{Sndef} in the case of a short memory process $X_t$ and deterministic sequence $T_i=\Delta i$ where $\Delta\in\mathbb{R}^+$.  When $X_t$ is a long memory linear process  and  $T_k$ are random,   \cite{ouldhaye2024asymptoticsirregularlyobservedlong}  established the limiting distribution of suitably normalized partial sums \eqref{Sndef} under both finite-mean {\rm(A1)} and infinite-mean renewal sampling {\rm(A2)}. In the finite-mean case, the limiting distribution is Gaussian. In contrast, under heavy-tailed renewal sampling, the limit may be a normal variance mixture whose random variance component is determined by the interplay between the heavy-tail index $\alpha$ and the long-memory parameter $d$.

The objective of the present paper is to extend this theory to
continuous-time L\'evy-driven moving-average processes. Although the limiting
distributions coincide with those obtained in the discrete-time framework, the
extension is far from routine. Several key ingredients of the discrete-time
proof are no longer available in the continuous-time setting.

The proof proceeds in two main steps. First, we establish in
Section~\ref{lindeberg2} a continuous-time Lindeberg proposition for stochastic
integrals of the form
$$
\int_{-\infty}^{\infty}f_n(s)\,dL_s,
$$
which provides  sufficient conditions for their asymptotic normality. We then
show that the normalized partial sums of the sampled process satisfy these
conditions. The remaining task is to analyze the random normalization. Much of
the corresponding discrete-time proof can be adapted to the present setting,
although new arguments are required in the finite-mean renewal case and in the
boundary case $\alpha=1-2d$. The main result is presented in Section~\ref{sec:main}. Section~\ref{parts} contains the remaining technical arguments required for its proof, together with the proofs of the auxiliary lemmas.
\section{Asymptotic normality of a sequence of a L\'evy-driven process}\label{lindeberg2}
We consider a stochastic process of the form
$$
Z_n:=\int_{-\infty}^\infty f_n(s)dL_s,\qquad n=1,2,\ldots,
$$
where $f_n\in L^2(\mathbb{R}).$ We establish a sufficient condition on $f_n$ to ensure the asymptotic normalilty of $Z_n$ (see Proposition \ref{lindeberg1}). The following lemma gives an explicit form of the characteristic function of the process $Z_n$ in terms of the characteristic 
 exponent $\Psi$ of $L_t$,   that satisfies that  for all $u,t\in\mathbb{R}$,
$$
\mathbb{E}\left(e^{iuL_t}\right)
=
e^{t\Psi(u)}. 
$$
Since $\mathbb{E}(L_1)=0$ and $\mathbb{E}(L_1)=1$, 
$$
\frac{\Psi(u)}{u^2}\to-\frac{1}{2},\qquad\textrm{as }u\to0.
$$
Also, from the proof of Theorem 8.1 of \cite{Sato1999}, we have
$$
\frac{\Psi(u)}{u^2}\to-\frac{\sigma^2}{2},\qquad\textrm{as }u\to\infty,
$$
where $\sigma^2$ is the diffusion coefficient of $L_t$ (i.e., $\sigma^2\ge0$ is  the variance of the Gaussian component of $L_t$). Since $\Psi$ is continuous, we can write
$$
\Psi(u)
=
-\frac{u^2}{2}\left(1+\epsilon(u)\right),
$$
where $\epsilon(u)\to0$ as $u\to0$ and $\epsilon(u)$ is bounded, say by $C$.
\begin{lem}\label{representation}
Let $g$ be an $L^2(\mathbb{R})$ real-valued function. Then
$$
\mathbb{E}\left[\exp\left(i\int_{-\infty}^\infty g(s)dL_s\right)\right]=\exp\left(\int_{-\infty}^\infty\Psi(g(s))ds\right).
$$
\end{lem}
The following proposition is a continuous-time analogue of the Lindeberg theorem  and will be the main tool for proving asymptotic normality of the normalized partial sums.
\begin{prop}\label{lindeberg1}

Let $L_s$ be a L\'evy process with mean zero, finite variance, and $\textrm{Var}(L_1)=1$ and let $f_n(s)$ be an $L^2(\mathbb{R})$-real valued  sequence of functions such
$$
\int_{-\infty}^\infty f_n^2(s)ds\to1,\qquad\textrm{as }n\to\infty,
$$
and  for each $\eta>0$
$$
\int_{\vert f_n(s)\vert>\eta}f_n^2(s)ds\to0\quad\textrm{as }n\to\infty.
$$
Then
\begin{equation}\label{normality}
\int_{-\infty}^\infty f_n(s)dL_s\overset{d}{\longrightarrow}\mathcal{N}(0,1).
\end{equation}
\end{prop}
\begin{proof}
    
Using Lemma \ref{representation},  the characteristic function of the left-hand side of \eqref{normality} can be written as
\begin{eqnarray}\label{conditional expectation}
&&\mathbb{E}\left[\exp\left(i\int_{-\infty}^\infty uf_n(s)\,dL_s\right)\right]
=\exp\left(\int_{-\infty}^\infty \Psi(uf_n(s))ds\right)\nonumber\\
&=&\exp\left(
-\frac{u^2}{2}
\int_{-\infty}^{\infty}
f_n(s)^2
\left(1+\epsilon(uf_n(s))\right)\,ds
\right)\nonumber\\
&=&
\exp\left(-\frac{u^2}{2}\int_{-\infty}^\infty f^2_n(s)ds\right)
\exp
\left\{
-\frac{u^2}{2}
\int_{-\infty}^{\infty}
f_n(s)^2\epsilon(uf_n(s))\,ds
\right\}
\nonumber\\
&=&\exp\left(-\frac{u^2}{2}\int_{-\infty}^\infty f^2_n(s)ds\right)\times\nonumber\\
&&
\exp
\left\{
-\frac{u^2}{2}
\int_{-\infty}^{\infty}
\left[f_n(s)^2\ind{\vert f_n(s)\vert\ge\eta}\epsilon(uf_n(s))+
f_n(s)^2\ind{\vert f_n(s)\vert<\eta}\epsilon(uf_n(s))\right]\,ds
\right\}\nonumber.
\end{eqnarray}
Now,
$$
\exp\left(-\frac{u^2}{2}\int_{-\infty}^\infty f^2_n(s)ds\right)\to e^{-u^2/2}\times1=e^{-u^2/2}\textrm{ as }n\to\infty.
$$
Also,
$$
\int_{-\infty}^{\infty}
f_n(s)^2\ind{\vert f_n(s)\vert\ge\eta}\epsilon(uf_n(s))ds\le C
\int_{-\infty}^{\infty}
f_n(s)^2\ind{\vert f_n(s)\vert\ge\eta}ds\to0,\qquad\textrm{as }n\to\infty,
$$
and
$$
\int_{-\infty}^{\infty}
f_n(s)^2\ind{\vert f_n(s)\vert\le\eta}\epsilon(uf_n(s))ds\le \left(\int_{-\infty}^\infty f^2_n(s)ds\right)
\underset{\vert v\vert\le\eta}{\sup}\vert\epsilon(uv)\vert\to0,\quad\textrm{as }\eta\to0,
$$
since 
$$
 \left(\int_{-\infty}^\infty f^2_n(s)ds\right)
$$
is bounded.
Then by the bounded convergence theorem, 
$$
\mathbb{E}\left[\exp\left(i\int_{-\infty}^\infty uf_n(s)\,dL_s\right)\right]\to e^{-u^2/2},\qquad\textrm{as }n\to\infty,
$$
the characteristic function of the standard normal distribution.

\end{proof}
\section{Convergence of partial sums of the sampled process}\label{sec:main}
The following theorem gives the asymptotic distribution of the normalized partial sums \eqref{Sndef} of the irregularly observed process $(Y_k)$. The asumptions of \cite{PhilippeRobetViano} to guarantee the $L^2$ stationarity of  $Y_k$ are satisfied ($X_t$ is $L^2$ stationary and $\Delta_i$ are i.i.d.)  It extends Theorem~1 of \cite{ouldhaye2024asymptoticsirregularlyobservedlong} to the continuous-time setting and shows that the limiting behaviour is governed by the same interaction between the renewal sampling scheme and the memory parameter $d$. First note that $S_n$ is not degenerated. Indeed, with $T=(\Delta_1,\ldots,\Delta_n)$, using \eqref{Sndef} and  the stochastic integral isometry, we have
\begin{equation}\label{conditional}
\Var(S_n)
=
\mathbb{E}(S_n^2)
=
\mathbb{E}\left(
\mathbb{E}\left(S_n^2\mid T\right)
\right)
=
\mathbb{E}\left(
\int_{-\infty}^{\infty}
\left(\sum_{i=1}^n a(T_i-s)\right)^2ds
\right),
\end{equation}
and since there exists $t_0>0$ such that for all for all $t\ge t_0$, $a(t)$ is non zero and of  constant sign,
$$
\int_{-\infty}^{\infty}
\left(\sum_{i=1}^n a(T_i-s)\right)^2ds\ge\int_{-\infty}^{-t_0}
\left(\sum_{i=1}^n a(T_i-s)\right)^2ds>0,\textrm{ almost surely}.
$$
\begin{thm}
Assume that the kernel $a$ is given by \eqref{a} with $0<d<1/2$, and that the renewal increments satisfy either {\rm(A1)} or {\rm(A2)}. Then, as $n\to\infty$,
$$
\left(\Var(S_n)\right)^{-1/2}S_n
\overset{d}{\longrightarrow}
\begin{cases}
N,
&
\textrm{if } \mathbb{E}(\Delta_1)<\infty,
\textrm{ or, under {\rm(A2)}, } \alpha=1
\textrm{ or } \alpha\le 1-2d,
\\[0.4cm]
\sqrt{Z(\alpha,d)}N,
&
\textrm{if, under {\rm(A2)}, } 1-2d<\alpha<1,
\end{cases}
$$
where $N$ is a standard normal random variable and the random variable $Z(\alpha,d)$ is defined by
$$
Z(\alpha,d)
=
C_{\alpha,1-2d}
\int_{[0,1]^2}|B_x-B_y|^{2d-1}\,dx\,dy,
$$
where $B_t$ is a L\'evy stable motion with parameter $\alpha$, $B_0=0$, and $B_t$ is nondecreasing in $t$, that is, $B_t$ is a stable subordinator. The process $B$ is independent of $N$, and for $0<r<\alpha$,
\begin{equation}\label{down1}
C_{\alpha,r}
=
(\alpha-r)(2\alpha-r)
\frac{\Gamma(r)}
{2\alpha\Gamma\left(\frac{r}{\alpha}\right)}.
\end{equation}
\end{thm}

\begin{proof} From \eqref{Sndef}, we can write
\begin{eqnarray*}
\left(\Var(S_n)\right)^{-1/2}S_n
=
R_n
\int_{-\infty}^\infty f_n(s)\,dL_s,
\end{eqnarray*}
where
\begin{equation}\label{normalized f_n}
f_n(s)
=
\left(
\int_{-\infty}^\infty
\left(
\frac{1}{n}\sum_{i=1}^n a(T_i-s)
\right)^2ds
\right)^{-1/2}
\left(
\frac{1}{n}\sum_{i=1}^n a(T_i-s)
\right).
\end{equation}
and
$$
R_n
=
\left(
\int_{-\infty}^\infty
\left(\sum_{i=1}^n a(T_i-s)\right)^2ds
\right)^{1/2}
\left(\Var(S_n)\right)^{-1/2}.
$$
The proof consists of showing 
\begin{eqnarray}\label{ratio}
R_n^2\overset{d}{\to}
\begin{cases}
1,
&
\textrm{if } \mathbb{E}(\Delta_1)<\infty,
\textrm{ or, under {\rm(A2)}, } \alpha=1
\textrm{ or } \alpha\le 1-2d,
\\[0.4cm]
Z(\alpha,d),
&
\textrm{ under {\rm(A2)} with } 1-2d<\alpha<1,
\end{cases}
\end{eqnarray}
and
\begin{equation}\label{self normalized}
\int_{-\infty}^\infty f_n(s)\,dL_s
\overset{d}{\longrightarrow}
\mathcal{N}(0,1),
\end{equation}
as well as  the asymptotic independence of 
$R_n$ and $\int_{-\infty}^\infty f_n(s)dL_s$  when $0<1-2d<\alpha<1$.\\
{\bf Proof of \eqref{ratio}.}
Using \eqref{a}, for $u\ge0$,  we can write
$$
a(u)=h(1+u)(1+u)^{d-1},
$$
where $h(1+u)\to c(d)$ given by \eqref{a}, as $u\to\infty$.
\begin{eqnarray}\label{equivalence}
\lefteqn{
\int_{-\infty}^\infty
\left(
\frac{1}{n}\sum_{i=1}^n a(T_i-s)
\right)^2ds
}\nonumber\\
&&=
\int_{-\infty}^\infty
\left(
\int_0^1
a(1+T_{[nx]}-s)\,dx
\right)^2ds
\nonumber\\
&&=
2\int_0^1\int_0^x\int_{-\infty}^\infty
h(1+T_{[nx]}-s)h(1+T_{[ny]}-s)(1+T_{[nx]}-s)^{d-1}
(1+T_{[ny]}-s)^{d-1}
\ind{s\le T_{[ny]}}\,ds\,dy\,dx
\nonumber\\
&&=
2\int_0^1\int_0^x\int_0^\infty h(1+T_{[nx]}-T_{[ny]}+u)h(1+u)
(1+T_{[nx]}-T_{[ny]}+u)^{d-1}
(1+u)^{d-1}\,du\,dy\,dx
\nonumber\\
&&=
\int_{[0,1]^2}(1+|T_{[nx]}-T_{[ny]}|)^{2d-1}\Bigg[
\int_0^\infty
(1+u)^{d-1}
\left(
\frac{1}{1+|T_{[nx]}-T_{[ny]}|}+u
\right)^{d-1}\times\nonumber\\
&&h\left((1+|T_{[nx]}-T_{[ny]}|)(1+u)\right)h\left(1+(1+|T_{[nx]}-T_{[ny]}|)u\right)
du
\Bigg]dx\,dy\nonumber\\
&&=\int_{[0,1]^2}g(1+|T_{[nx]}-T_{[ny]}|)(1+|T_{[nx]}-T_{[ny]}|)^{2d-1}dx\,dy,
\end{eqnarray}
where for all $v\ge1$, 
\begin{eqnarray}\label{finite d}
g(v)&=&\int_0^\infty h(1+vu)h(v(1+u))\left(\frac{1}{v}+u\right)^{d-1}(1+u)^{d-1}du.
\end{eqnarray}
From \eqref{equivalence} and \eqref{conditional} with Fubini, we can write
$$
R_n^2
=
\frac{
\displaystyle\int_{[0,1]^2}
g(1+|T_{[nx]}-T_{[ny]}|)(1+|T_{[nx]}-T_{[ny]}|)^{2d-1}\,dx\,dy
}{
\displaystyle\int_{[0,1]^2}
\mathbb{E}
\left[g(1+|T_{[nx]}-T_{[ny]}|)
(1+|T_{[nx]}-T_{[ny]}|)^{2d-1}
\right]dx\,dy
}.
$$
This representation of $R_n^2$ allows us to adapt the proof established in
\cite{ouldhaye2024asymptoticsirregularlyobservedlong} for the discrete-time
process $X_t$. As in the discrete-time setting, the analysis splits
into five cases according to the interplay between the renewal process and the memory parameter $d$.
A careful examination of the proof of Theorem~1 in
\cite{ouldhaye2024asymptoticsirregularlyobservedlong} shows that the arguments
carry over unchanged in three of these cases. Only the cases
$\mathbb{E}(\Delta_1)<\infty$ and $\alpha=1-2d$ with
$b_n^{-\alpha}$ not summable require new arguments. The sequence $b_n$  is regularly varying at infinity  with index $1/\alpha$ and satisfies
$$
\frac{T_{[nx]}}{b_n}
\overset{D[0,1]}{\longrightarrow}
B_x,
$$
\cite[see for example][Ch. 7, Corollary 7.1 and page 247] {resnick2007heavy}. The completion of the  proof of \eqref{ratio}, covering these two  cases, is given in the Subsection \ref{technical}. \\
\medskip
{\boldmath\textbf{Proof of \eqref{self normalized}}}
With $f_n$ in \eqref{normalized f_n}, writing 
$$
\mathbb{E}\left[\exp\left(iu\int_{-\infty}^\infty f_n(s)dL_s\right)\right]=\mathbb{E}\left(\mathbb{E}\left[\exp\left(iu\int_{-\infty}^\infty f_n(s)dL_s\right)\Bigg\vert T\right]\right),
$$
and using Proposition \ref{lindeberg1}, \eqref{self normalized} follows from the following Lemma.
\begin{lem}\label{in proba con}
Let $\eta>0$, and  let $f_n$ be as defined in \eqref{normalized f_n}. Then
    \begin{equation}\label{unif fn}
\int_{\vert f_n(s)\vert>\eta}\vert f_n(s)\vert
\overset{P}{\to}0,\qquad\textrm{as }n\to\infty.
\end{equation}
\end{lem}
Finally, the proof of the asymptotic independence of
\[
R_n \quad \text{and} \quad \int_{-\infty}^{\infty} f_n(s)\,dL_s
\]
when \(0<1-2d<\alpha<1\) can be carried out in essentially the same way as in the discrete-time case considered in \cite{ouldhaye2024asymptoticsirregularlyobservedlong}, as it relies solely on a conditional version of \eqref{self normalized}, namely that for every \(x\in\mathbb{R}\),
\[
P\left(\int_{-\infty}^{\infty} f_n(s)\,dL_s \le x \,\Big|\, T\right)
\overset{P}{\longrightarrow}
P(N\le x),
\]
which can be established  in the same way as in the proof of Proposition \ref{lindeberg1}.
\end{proof}
\section{Technical proofs}\label{parts}
\subsection{Completion of the proof of \eqref{ratio}.}\label{technical}
We start with the following Lemma, which will be proven in the subsection \ref{proof of Tech lemma}.
\begin{lem}\label{function g}
The function $g$ in \eqref{finite d} is well defined, bounded, and as $v\to\infty$, 
\begin{equation}\label{delta d}
g(v)\to c^2(d)\int_0^\infty u^{d-1}(1+u)^{d-1}du=c^2(d)\frac{\Gamma(d)\Gamma(1-2d)}{\Gamma(1-d)}:=\delta(d).
\end{equation}
\end{lem}
Let us now go back to the end of the proof of \eqref{ratio}. \\
\noindent{\boldmath\textbf{Case of $\mathbb{E}(\Delta_1)<\infty$.}} The  proof will consist in showing that, properly normalized, the numerator of $R_n^2$ is $P\times\lambda^2$ uniformly integrable, where $\lambda$ is the Lebesgue probability measure on the unit interval.  \\
From \cite{SHI20101260}, for $0<r<1$,
$$
\mathbb{E}\left[(1+T_k)^{-r}\right]
\le
ck^{-r},
$$
and hence
$$
\mathbb{E}
\left[
\left(
\frac{1+|T_{[nx]}-T_{[ny]}|}
{1+|[nx]-[ny]|}
\right)^{-r}
\right]
$$
is uniformly bounded, say by $M$, in $n\ge1$ and $(x,y)\in[0,1]^2$. Using the fact that $1+|[nx]-[ny]|\ge n|x-y|$, we have
\begin{eqnarray*}
\lefteqn{
\int_{[0,1]^2}
\mathbb{E}
\left[
\left(
\frac{1+|T_{[nx]}-T_{[ny]}|}
{n}
\right)^{-r}
\right]dx\,dy
}\\
&&=
\int_{[0,1]^2}
\mathbb{E}
\left[
\left(
\frac{1+|T_{[nx]}-T_{[ny]}|}
{1+|[nx]-[ny]|}
\right)^{-r}
\right]
\left(
\frac{1+|[nx]-[ny]|}{n}
\right)^{-r}
dx\,dy\\
&&\le
M
\int_{[0,1]^2}
|x-y|^{-r}\,dx\,dy
<
\infty.
\end{eqnarray*}
Therefore
$$
g(1+|T_{[nx]}-T_{[ny]}|)\left(
\frac{1+|T_{[nx]}-T_{[ny]}|}
{n}
\right)^{2d-1}
$$
is $P\times\lambda^2$ uniformly integrable. Hence, 
$$
\int_{[0,1]^2}
\mathbb{E}
\left[
g(1+|T_{[nx]}-T_{[ny]}|)\left(
\frac{1+|T_{[nx]}-T_{[ny]}|}
{n}
\right)^{2d-1}
\right]
dx\,dy
\to
\delta(d)\frac{\left(\mathbb{E}(\Delta_1)\right)^{2d-1}}
{d(2d+1)}.
$$
Moreover, the uniform integrability above also allows to apply  \cite{CREMERS1986305} to obtain that,  as $n\to\infty$,
\begin{eqnarray}\label{finite moment}
\lefteqn{
\int_{[0,1]^2}
g(1+|T_{[nx]}-T_{[ny]}|)\left(
\frac{1+|T_{[nx]}-T_{[ny]}|}
{n}
\right)^{2d-1}
dx\,dy
}\nonumber\\
&&\overset{P}{\to}
\delta(d)\left(\mathbb{E}(\Delta_1)\right)^{2d-1}
\int_{[0,1]^2}|x-y|^{2d-1}\,dx\,dy\nonumber\\
&&
=\delta(d)
\frac{\left(\mathbb{E}(\Delta_1)\right)^{2d-1}}
{d(2d+1)}
<
\infty.
\end{eqnarray}
Thus \eqref{ratio} holds.\\
\medskip
\noindent{\boldmath\textbf{Case of $\alpha=1-2d$ and $b_n^{-\alpha}$  not summable.}}
Let
\begin{equation}\label{V_n}
V_n
=
\sum_{k=1}^n b_k^{-\alpha}
\to\infty,
\qquad n\to\infty.
\end{equation}
Then we prove that
\begin{equation}\label{non summable b_n}
\frac{n}{V_n}
\int_{[0,1]^2}
g(1+|T_{[nx]}-T_{[ny]}|)\left(1+|T_{[nx]}-T_{[ny]}|\right)^{-\alpha}
dx\,dy
\overset{L^2}{\longrightarrow}
C,
\end{equation}
for some constant $C>0$ and  \eqref{ratio} then holds.
As in discrete-time case, the tool for establishing \eqref{non summable b_n} is to show that, as $n\to\infty$,
\begin{equation}\label{unif variance}
\sup_{k\le n}\Var(Z'_{n,k})
\to0,
\end{equation}
where
$$
Z'_{n,k}
=
\frac{1}{V_n}
\sum_{\ell=1}^{n-k}
g(1+T_{k+\ell}-T_k)\left(1+T_{k+\ell}-T_k\right)^{-\alpha}.
$$
The difficulty in establishing \eqref{unif variance} stems from the fact that, in continuous time, the renewal
increments satisfy only $\Delta_i>0$, whereas the discrete-time setting
automatically imposes $\Delta_i\ge1$. Consequently, 
we decompose $Z'_{n,k}$ as follows
\begin{eqnarray}\label{decomp of Z}\lefteqn{
Z'_{n,k}=
\frac{1}{V_n}
\sum_{\ell=1}^{n-k}
g(1+T_{k+\ell}-T_k)\left(1+T_{k+\ell}-T_k\right)^{-\alpha}
\mathbf 1_{\{N_{k,\ell}^{(n)}\ge\ell/2\}}}\nonumber\\
&&+
\frac{1}{V_n}
\sum_{\ell=1}^{n-k}
g(1+T_{k+\ell}-T_k)\left(1+T_{k+\ell}-T_k\right)^{-\alpha}
\mathbf 1_{\{N_{k,\ell}^{(n)}<\ell/2\}},
\end{eqnarray}
where
$$
N_{k,\ell}^{(n)}=\sum_{i=k+1}^{k+\ell}\ind{\Delta_i>a},
$$
for certain $a>0$ such that $P(\Delta_i>a)\ge3/4.$ This is possible since $P(\Delta_i>0)=1$ by assumption.
For the first term of $Z'_{n,k}$ in \eqref{decomp of Z}, we have
\begin{eqnarray*}\lefteqn{
\textrm{Var}\left(\frac{1}{V_n}
\sum_{\ell=1}^{n-k}
g(1+T_{k+\ell}-T_k)\left(1+T_{k+\ell}-T_k\right)^{-\alpha}
\mathbf 1_{\{N_{k,\ell}^{(n)}\ge\ell/2\}}\right)}\\
&&
=(V_n)^{-2}\int_0^{1-\frac{k}{n}}\int_0^{1-\frac{k}{n}}\\
&&\Bigg[\mathbb{E}\left(g(1+T_{[ns]})(1+T_{[ns]})^{-\alpha}\ind{T_{[ns]\ge[ns]/2}}g(1+T_{[nt]})(1+T_{[nt]})^{-\alpha}\ind{T_{[nt]\ge[nt]/2}}\right)\\
&&-\mathbb{E}\left(g(1+T_{[ns]})(1+T_{[ns]})^{-\alpha}\ind{T_{[ns]\ge[ns]/2}}\right)\mathbb{E}\left((1+T_{[nt]})(1+T_{[nt]})^{-\alpha}\ind{T_{[nt]\ge[nt]/2}}\right)\Bigg]dsdt
\end{eqnarray*}
Thanks to the indicator function, the same reasoning as in \cite{ouldhaye2024asymptoticsirregularlyobservedlong} works
to show that, the quantity above converges to zero, uniformly in $k\le n$.  The second term of $Z'_{n,k}$ in \eqref{decomp of Z} is bounded by
\begin{eqnarray*}
\lefteqn{
\frac{1}{V^2_n}
\mathbb{E}
\left(
\sum_{\ell=1}^{n-k}
\left(1+T_{k+\ell}-T_k\right)^{-\alpha}
\mathbf 1_{\{N_{k,\ell}^{(n)}<\ell/2\}}
\right)^2
}\\
&&\le
2\frac{(\log(V_n))^2}{V^2_n}
+
2\frac{1}{V^2_n}
\mathbb{E}
\left[
\left(
\sum_{\ell=\log(V_n)}^{n-k}
\mathbf 1_{\{N_{k,\ell}^{(n)}<\ell/2\}}
\right)^2
\right].
\end{eqnarray*}
Using the Chernoff bound (see, for example, \cite{DubhashiPanconesi2009}, Theorem~1.1), namely that, if $X\sim\mathrm{Binomial}(n,p)$, then
\[
P(X\le (1-\delta)np)\le \exp\left(-\frac{\delta^2np}{2}\right),\qquad 0\le\delta\le1,
\]
and noting that $N_{k,\ell}^{(n)}$ is binomially distributed with parameters $\ell$ and $p:=P(\Delta_1>a)\ge3/4$, we choose $\delta=1-\frac{1}{2p}$ to obtain
\[
P\left(N_{k,\ell}^{(n)}\le \frac{\ell}{2}\right)
\le
\exp\left(-\frac12\left(1-\frac{1}{2p}\right)^2\ell p\right)
\le
\exp\left(-\frac{\ell}{24}\right).
\]
As $n\to\infty$,
\begin{eqnarray*}
\lefteqn{
\frac{1}{V_n}
\mathbb{E}
\left[
\left(
\sum_{\ell=\log(V_n)}^{n-k}
\mathbf 1_{\{N_{k,\ell}^{(n)}<\ell/2\}}
\right)^2
\right]
}\\
&&\le
\frac{1}{V^2_n}
\left(
\sum_{\ell=\log(V_n)}^{n}
P\left(N_{k,\ell}^{(n)}<\ell/2\right)
+
2\sum_{\ell=\log(V_n)}^{n}
\sum_{\ell'=\ell}^{n}
P\left(N_{k,\ell'}^{(n)}<\ell'/2\right)
\right)\\
&&=
o\left(
\frac{1}{V^2_n}
\right)
\to0
\end{eqnarray*}
and hence the second term of \eqref{decomp of Z} is bounded by $(\log V_n)^2/V_n^2\to0$ as $n\to\infty$ since $V_n\to\infty.$
This completes the proof of \eqref{unif variance}.
\subsection{Proofs of the Lemmas}
{\bf Proof of Lemma \ref{representation}.}
\begin{proof}
Note that this  Lemma  was proved for left-continuous functions over closed intervals in \cite{ContTankov2003} (Lemma 15.1).
First let assume $g$ simple function taking  values $c_j$ on intevals $(a_j,b_j]$, $j=1,\ldots,m,$ $m\ge1$ then as in the proof of the Lemma 15.1 of \cite{ContTankov2003}, using the independence of the increments of $L_s$, we obtain that
\begin{eqnarray*}
\mathbb{E}\left[\exp\left(i\int_{-\infty}^\infty g(s)\,dL_s\right)\right]
&=&\prod_{i=1}^m\mathbb{E}\left(e^{ic_j(L_{b_j-a_j})}\right)=\exp\left(\sum_{j=1}^m(b_j-a_j)\Psi(c_j)\right)\\
&=&\exp\left(\int_{-\infty}^\infty \Psi(g(s))ds\right).
\end{eqnarray*}
If $g$ is in $L^2(\mathbb{R})$ then it can be approximated by a left-continuous simple function $g_n$. Actually, from Theorem 2.41 of \cite{Folland1999}, the class of simple functions taking values $c_j$ on intervals $(a_j,b_j]$ is dense in $L^2$ (the Theorem is stated for $L^1(\mathbb{R})$  but the proof is the same for $L^2(\mathbb{R})$), so there exists a sequence  of simple left continuous functions $g_n$
 such that
 $$
\int_{-\infty}^\infty (g_n(s)-g(s))^2ds\to0,\qquad\textrm{as }n\to\infty.
$$
Writing $\Psi(g_n(s))=g^2_n(s)h(g_n(s))$ and $\Psi(g(s))=g^2(s)h(g(s))$, where $h(x)=\Psi(x)/x^2$, so that $h$ is bounded, we have $\Psi(g_n)\to\Psi(g)$ in $L^1(\mathbb{R})$, as $\vert\Psi(x)\vert\le Cx^2$ so that $\Psi(g_n(s))$ is uniformly integrable and $\Psi(g(s))$ is integrable. In particular, we get
\begin{equation}\label{two levels}
\int_{-\infty}^\infty \Psi(g_n(s))ds\to\int_{-\infty}^\infty\Psi(g(s))ds,\qquad\textrm{as }n\to\infty.
\end{equation}
\begin{eqnarray*}\lefteqn{
\mathbb{E}\left[\exp\left(i\int_{-\infty}^\infty g_n(s)\,dL_s\right)-\exp\left(i\int_{-\infty}^\infty g(s)\,dL_s\right)\right]}\\
&&=\mathbb{E}\left[\exp\left(i\int_{-\infty}^\infty g(s)\,dL_s\right)\left\{\exp\left(i\int_{-\infty}^\infty \left(g_n(s)-g(s)\right)\,dL_s\right)-1\right\}\right]\to0\textrm{ as }n\to\infty.
\end{eqnarray*}
Hence we obtain that
$$
\exp\left(\int_{-\infty}^\infty \Psi(g_n(s))ds\right)\to\mathbb{E}\left[\exp\left(i\int_{-\infty}^\infty g(s)\,dL_s\right)\right],
$$
which, combined with \eqref{two levels}, gives
$$
\mathbb{E}\left[\exp\left(i\int_{-\infty}^\infty g(s)\,dL_s\right)\right]=\exp\left(\int_{-\infty}^\infty\Psi(g(s))ds\right).
$$
\end{proof}\noindent
{\bf Proof of Lemma \ref{in proba con}.}
\begin{proof}
It follows the usual truncation argument. We decompose the kernel into a bounded part and a remainder and estimate the corresponding contributions separately.
Let $A,B$ be positive such that $\vert h(u)\vert<B$ for $u\ge A.$ This is possible since $h(u)\to c(d)$ as $u\to\infty.$
With the convention that $h(u)=0$ for $u<0$, write
$$
h(u)=h(u)\ind{h(u)\le B}+h(u)\ind{h(u)>B}:=h^{(B)}(u)+h^{(R)}(u).
$$
Let
$$
U^2_n=\int_{-\infty}^\infty
\left(
\frac{1}{n}\sum_{i=1}^n a(T_i-s)
\right)^2ds.
$$
According to \eqref{equivalence}, as $n\to\infty$, $U_n^2$ is the numerator of
$R_n^2$. As established in the proof of \eqref{ratio}, 
\begin{equation*}
U_n/d_n\overset{d}{\longrightarrow}C,
\end{equation*}
where $C$ is either a positive random variable or a positive constant,
depending on the interplay between the renewal process and the memory
parameter $d$, and $d_n$ is given by
\begin{equation}\label{five guys}
d_n=\begin{cases}
n^{d-1/2},&\textrm{ if }\mathbb{E}(\Delta_1)<\infty,\\
b_n^{d-1/2},&\textrm{ if }2d-1<\alpha<1,\\
h_n^{d-1/2},&\textrm{ if }\alpha=1,\\
n^{1/2},&\textrm{ if }\alpha\le1-2d\textrm{ and }b_n^{-\alpha}\textrm{ is summable},\\
(n/V_n)^{1/2},&\textrm{ if }\alpha=1-2d,\textrm{ and }b_n^{-\alpha}\textrm{ is not summable},
\end{cases}
\end{equation}
where $h_n$ is regularly varying with index $1/\alpha$, and $V_n$ is defined in \eqref{V_n}.
Without loss of generality, assume $h^{(B)}$ and $h^{(R)}$ are nonnegative. 
\begin{eqnarray*}\lefteqn{
f_n(s)=U_n^{-1}\int_0^1 h^{(B)}(1+(T_{[nx]}-s)(1+T_{[nx]}-s)^{d-1}ds}\\
&&+U_n^{-1}\int_0^1 h^{(R)}(1+T_{[nx]}-s)(1+T_{[nx]}-s)^{d-1}ds.
\end{eqnarray*}
Using the basic facts that $(a+b)^2\le 2(a^2+b^2)$ and that for $f,g\ge0$
$$
\int_{f+g\ge\eta}(f(s)+g(s))ds\le2\int_{f\ge\eta/2}f(s)ds+2\int_{g\ge\eta/2}g(s)ds, 
$$
and that
$$
\int_{f\ge\eta}f^2(s)ds\le\eta^{-1}\int_{-\infty}^\infty f^3(s)ds,
$$
we can write
\begin{eqnarray*}\lefteqn{
\int_{\vert f_n(s)\vert>\eta}f^2_n(s)ds}\\
&&\le 8\eta^{-1}U_n^{-3}\int_{-\infty}^\infty\left(
\int_0^1 h^{(B)}(1+T_{[nx]}-s)(1+T_{[nx]}-s)^{d-1}dx\right)^3ds\\
&&+4U_n^{-2}\int_{-\infty}^\infty\left(
\int_0^1 h^{(R)}(1+T_{[nx]}-s)(1+T_{[nx]}-s)^{d-1}dx\right)^2ds\\&&
:=I^{(B)}(n)+I^{(R)}(n).
\end{eqnarray*}
With  variable changes $u=T_{[nz]}-s$ and then  $u=(1+T_{[nx]}-T_{[nz]})v$ in the quantities below, we can write
\begin{eqnarray*}\lefteqn{
\mathbb{E}\left(\int_{-\infty}^\infty\left(
\int_0^1 h^{(B)}(1+T_{[nx]}-s)(1+T_{[nx]}-s)^{d-1}dx\right)^3ds\right)}\\
&&=6\int_{0<z<y<x<1}\mathbb{E}\Bigg(\int_0^\infty
h^{(B)}(1+T_{[nx]}-T_{[nz]}+u)h^{(B)}(1+T_{[ny]}-T_{[nz]}+u)h^{(B)}(1+u)\\
&&(1+T_{[nx]}-T_{[nz]}+u)^{d-1}(1+T_{[ny]}-T_{[nz]}+u)^{d-1}(1+u)^{d-1}du
\Bigg)dzdydx\\
&&
\le 6B^3\int_{0<z<y<x<1}\Bigg(\int_0^\infty(1+v)^{d-1}v^{d-1}dv\Bigg)\\
&&\mathbb{E}\left[(1+T_{[nx]}-T_{[nz]})^{2d-1}(1+T_{[ny]}-T_{[nz]})^{d-1}\right]dzdydx\\
&&\le 6B^3\left(\int_0^\infty(1+v)^{d-1}v^{d-1}dv\right)\times\\
&&\int_0^1\int_0^x\int_0^y\mathbb{E}\left((1+T_{[nx]}-T_{[ny]})^{2d-1}\right)\mathbb{E}\left((1+T_{[ny]}-T_{[nz]})^{d-1}\right)dzdydx.
\end{eqnarray*}
Using the fact that for all $0<r<1$, $\mathbb{E}(1+T_k)^{-r}\le Cb_k^{-r}$ (see Lemma 1 (i) of \cite{ouldhaye2024asymptoticsirregularlyobservedlong} when $\Delta_1$ has heavy tail ($A_2$) and \cite{SHI20101260} when $\mathbb{E}(\Delta_1)<\infty$, $(A_1$), with some small $\delta>0$, we can bound the last integral by 
\begin{eqnarray*}\lefteqn{
    n^{\delta-3}C\sum_{1\le k<j<i\le n}(i-j)^{\frac{2d-1}{\alpha}}(j-k)^{\frac{d-1}{\alpha}}}\\
    &&\le Cn^{\delta-2}\left(\sum_{h=1}^nh^{\frac{2d-1}{\alpha}}\right)^2.
\end{eqnarray*}
Writing 
$$
U_n^{-3}=\left(\frac{U_n}{d_n}\right)^{-3}d_n^{-3}
$$
and going over all the five possible values of $d_n$ in \eqref{five guys}, we obtain that for each fixed $B>0$, $I^{(B)}(n)\to0$ as $n\to\infty.$  Let us now check the remaining quantity $I^{(R)}(n)$. Using the fact that $ab\le(a^2+b^2)/2$, we have
\begin{eqnarray}\label{two terms}\lefteqn{
I^{(R)}(n)}\nonumber\\
&&\le\frac{8}{U_n^2}\int_0^1\int_0^x(1+T_{[nx]}-T_{[ny]})^{d-1}
\left(\int_0^Ah^{(R)}(1+T_{[nx]}-T_{[ny]}+u)h^{(R)}(1+u)du\right)dydx\nonumber\\
&&\le8\left(\int_0^Ah^2(1+u)\ind{h(1+u)\ge B}du\right)\frac{1}{U_n^2}\int_{[0,1]^2}(1+\vert T_{[nx]}-T_{[ny]}\vert)^{d-1}dxdy.
\end{eqnarray}
Now,
$$
\mathbb{E}\left(\frac{1}{d_n^2}\int_{[0,1]^2}(1+\vert T_{[nx]}-T_{[ny]}\vert)^{d-1}dxdy\right)
$$
converges to zero when $1-d<\alpha\le1$, and bounded when $\alpha\le1-d$ and therefore it is bounded in all cases. The first bracket in \eqref{two terms} converges to zero as $B\to\infty$ since $h^2$ is integrable over $(0,A)$. Now again, using the fact that $d_n/U_n$ converges in probability and therefore bounded in probability, we conclude that  
$$
\underset{n\ge1}{\sup}P\left(I_n(R)\ge\epsilon)\right)\to0\textrm{ as  }B\to\infty,
$$
which in turns, concludes the proof that  that for any $\eta>0$, as $n\to\infty,$
$$
\limsup\int_{\vert f_n(s)\vert>\eta}f^2_n(s)ds=0
$$
in Probability.
\end{proof}\noindent
{\bf Proof of Lemma \ref{function g}}\label{proof of Tech lemma}
\begin{proof}
By \eqref{finite 2}, $h(1+t)$ is square integrable on every finite interval. Let $A>0$ be such that $h(1+t)\le C$ for all $t\ge A$. Putting $t=uv$, we get for $v\ge1$,
\begin{eqnarray*}\lefteqn{
\int_0^\infty \vert h(1+vu)h(v(1+u))\vert\left(\frac{1}{v}+u\right)^{d-1}(1+u)^{d-1}du
}\\
   &&\le\begin{cases}
       C\left[\displaystyle\int_0^A\vert h(1+t)\vert dt+C\frac{\Gamma(d)\Gamma(1-2d)}{\Gamma(1-d)}\right],&\textrm{if }v\ge A,\\\\
\displaystyle\int_0^{2A}h^2(1+t)dt+C^2\frac{\Gamma(d)\Gamma(1-2d)}{\Gamma(1-d)},&\textrm{if }v\le A,
   \end{cases}
\end{eqnarray*}
showing that $g(v)$ is defined and  uniformly bounded. Also, for $v\ge A$,
\begin{eqnarray*}
g(v)&=&\int_0^\infty h(1+vu)h(v(1+u))\left(\frac{1}{v}+u\right)^{d-1}(1+u)^{d-1}\ind{uv\le A}du\\
&&+\int_0^\infty  h(1+vu)h(v(1+u))\left(\frac{1}{v}+u\right)^{d-1}(1+u)^{d-1}\ind{uv\ge A}du.
\end{eqnarray*}
The first integral is bounded in absolute value by
$$
v^{-d}C\int_0^A \vert h(1+t)\vert dt\to0\qquad\textrm{as }v\to\infty.
$$
For fixed $u>0$, and as $v\to\infty$,  the integrand in the 2nd integral converges to 
$$
c^2(d)u^{d-1}(1+u)^{d-1}
$$
and is bounded by
$$
C^2u^{d-1}(1+u)^{d-1}
$$
which is integrable, 
so that by the dominated convergence Theorem, the 2nd integral converges to $\delta(d)$ defined in \eqref{delta d}. This concludes  the proof of all properties of $g$ defined in \eqref{finite d}.
\end{proof}
\bibliographystyle{apalike}
\bibliography{reference}

\end{document}